\documentclass[a4]{amsart}

\usepackage{amssymb}
\usepackage[all]{xy}
\usepackage{amsmath, amsfonts, amsthm , amssymb} 

\input xypic 
\usepackage{tikz} 

\newtheorem{theorem}{Theorem}[section]
\newtheorem{proposition}[theorem]{Proposition}
\newtheorem{lemma}[theorem]{Lemma}
\newtheorem{corollary}[theorem]{Corollary}

\theoremstyle{definition}

\newtheorem{example}[theorem]{Example}

\newcommand{\uxa}{\ensuremath{(\underline{X},\underline{A})}}

\newcommand{\zk}{\ensuremath{\mathcal{Z}_{K}}}

\newcommand{\conn}{\ensuremath{\#}}  

\newcounter{bean}

\newcommand{\namedright}[3]{\ensuremath{#1\stackrel{#2}
 {\longrightarrow}#3}}
\newcommand{\nameddright}[5]{\ensuremath{#1\stackrel{#2}
 {\longrightarrow}#3\stackrel{#4}{\longrightarrow}#5}}
\newcommand{\namedddright}[7]{\ensuremath{#1\stackrel{#2}
 {\longrightarrow}#3\stackrel{#4}{\longrightarrow}#5
  \stackrel{#6}{\longrightarrow}#7}} 

\newcommand{\larrow}{\relbar\!\!\relbar\!\!\rightarrow}
\newcommand{\llarrow}{\relbar\!\!\relbar\!\!\larrow}

\newcommand{\llnamedright}[3]{\ensuremath{#1\stackrel{#2}
 {\llarrow}#3}}
\newcommand{\llnameddright}[5]{\ensuremath{#1\stackrel{#2}
 {\llarrow}#3\stackrel{#4}{\llarrow}#5}}

\newcommand{\qqed}{\hfill\Box}

\begin{document}
\title[Moment-angle manifolds]{Moment-angle manifolds associated to neighbourly triangulations of spheres} 
\author{Amaranta Membrillo Solis} 
\address{School of Mathematical Sciences, Queen Mary University of London, London E1 4NS, UK} 
\email{i.a.membrillosolis@qmul.ac.uk}
\author{Stephen Theriault}
\address{School of Mathematical Sciences, University of Southampton, Southampton
SO17 1BJ, UK} 
\email{S.D.Theriault@soton.ac.uk}

\subjclass[2020]{Primary 55P15, 57S12}
\keywords{moment-angle manifold, neighbourly triangulation, connected sum, homotopy type}

\begin{abstract} 
We show that a moment-angle manifold associated to a neighbourly triangulation of an odd 
dimensional sphere is homotopy equivalent to a connected sum of products of two spheres, 
resolving a problem of Buchstaber and Panov. The methods are entirely homotopy theoretic, 
allowing for an extension to a corresponding result in the case of generalized moment-angle 
manifolds. 
\end{abstract}

\maketitle 

\section{Introduction} 

Moment-angle complexes are of fundamental importance in toric topology and play significant  
roles in other areas of mathematics. To each finite simplicial complex $K$ there is an associated 
topological space called a moment-angle complex $\zk$. The cohomology of $\zk$ is the 
commutator subalgebra of the 
Stanley-Reisner face ring of $K$. The space $\zk$ is homotopy equivalent to the complement 
of the complex coordinate subspace arrangement associated to $K$. If $K$ is a triangulation 
of a sphere then $\zk$ is a manifold equipped with a canonical torus action, and this manifold 
structure is smooth if $K$ is the underlying complex of a complete simplicial fan~\cite{BP}. 
Specializing further, if~$K$ is the boundary of the dual of a simple convex polytope then $\zk$ 
admits a geometric model as an intersection of real quadrics in $\mathbb C^{m}$ \cite{LdM}. 

An important problem is to understand the homotopy types of moment-angle complexes 
in general and moment-angle manifolds in particular. In this paper we identify the homotopy 
type of moment-angle manifolds associated to neighbourly triangulations of odd dimensional 
spheres, resolving a problem posed by Buchstaber and Panov. Our methods are entirely 
homotopy theoretic, so they also extend to generalized moment-angle manifolds, 
which are special cases of polyhedral products. 

To describe our results more precisely, we set some notation and define terms. 
Let $K$ be a simplicial complex on the vertex set $[m]$.  For $1\leq i\leq m$,
let $(X_{i},A_{i})$ be a pair of pointed $CW$-complexes, where $A_{i}$ is 
a pointed subspace of~$X_{i}$. Let $\uxa=\{(X_{i},A_{i})\}_{i=1}^{m}$ be 
the sequence of $CW$-pairs. For each simplex (face) $\sigma\in K$, let 
$\uxa^{\sigma}$ be the subspace of $\prod_{i=1}^{m} X_{i}$ defined by
\[\uxa^{\sigma}=\prod_{i=1}^{m} Y_{i}\qquad
       \mbox{where}\qquad Y_{i}=\left\{\begin{array}{ll}
                                             X_{i} & \mbox{if $i\in\sigma$} \\
                                             A_{i} & \mbox{if $i\notin\sigma$}.
                                       \end{array}\right.\]
The \emph{polyhedral product} determined by \uxa\ and $K$ is
\[\uxa^{K}=\bigcup_{\sigma\in K}\uxa^{\sigma}\subseteq\prod_{i=1}^{m} X_{i}.\] 
If all pairs $(X_{i},A_{i})$ are the same the polyhedral product is written as $(X,A)^{K}$. 
The \emph{moment-angle complex} is the polyhedral product $\mathcal{Z}_{K}=(D^{2},S^{1})^{K}$. 
A \emph{generalized moment-angle complex} is the polyhedral product $(D^{k},S^{k-1})^{K}$ 
for some integer $k\geq 2$. 

When $K$ is a triangulation of a sphere, $\zk$ is a topological manifold, known as a 
\emph{moment-angle manifold}. There are few cases where the homotopy type of 
a moment-angle manifold is known and in those cases $\zk$ is smooth and the diffeomorphism 
type is known. This comes from studying moment-angle manifolds from the viewpoint 
of intersections of quadrics~\cite{LMV,Me,Bo,BM}. McGavran~\cite{M} in the case of a $2$-simplex 
and Bosio and Meersseman~\cite{BM} in the case of a $3$-simplex considered convex polytopes 
obtained from the simplex by successive vertex cuts, and showed that the moment-angle manifold 
associated to the dual of the boundary of such a polytope has the diffeomorphism type of a 
connected sum of products of two spheres. Gitler and López de Medrano extended this to 
even dimensional polytopes that are dual neighbourly \cite{GLdM} and odd dimensional polytopes 
with additional conditions. There is an isolated instance of a moment-angle manifold known 
to be diffeomorphic to a connected sum involving a product of three or more spheres~\cite{I}, while 
recent work of Kovyrshina and Panov~\cite{KP} suggests more instances like this should exist. 

A simplicial complex \(K\) is \(\ell\)-neighbourly if every subset of \(\ell+1\) vertices spans a 
simplex of \(K\). A triangulation \(K\) of \(S^d\) is polytopal if it is combinatorially equivalent 
to the boundary of a simplicial convex \((d+1)\)-polytope. The family of neighbourly 
non-polytopal triangulations of \(S^d\) is much larger than the neighbourly polytopal family, and this 
difference already appears in classifications of triangulated spheres with few vertices. For 
example, for triangulations of $S^{3}$: there are $4$ neighbourly triangulations on $8$ vertices, of 
which $3$ are polytopal~\cite{Ba}; there are \(50\) neighbourly triangulations on \(9\) vertices, 
of which only \(23\) are polytopal \cite{AS}; and there are \(3540\) neighbourly triangulations on $10$ 
vertices, of which only between \(333\) and \(432\) are polytopal \cite{A}. Asymptotically, for a 
fixed dimension $d$ and letting the number of vertices increase, almost all neighbourly triangluations 
of $S^{d}$ are non-polytopal: for \(d\ge 5\), the number of \(\lfloor d/2\rfloor\)-neighbourly \((d-1)\)-spheres on \(n\) vertices is at least \(2^{\Omega\!\left(m^{\lfloor (d-1)/2\rfloor}\right)}\), whereas the total number of \(d\)-polytopes with \(m\) vertices is only \(2^{\Theta(m\log m)}\) \cite{NZ}.

The Gitler and López de Medrano  result shows that if $K$ is a neighbourly polytopal triangulation 
of an odd dimensional sphere then $\zk$ is diffeomorphic to a connected sum of products of two spheres. 
Since neighbourly non-polytopal triangulations are much more numerous, it is natural to ask 
what happens in those cases. We develop new methods to give a homotopy theoretic answer. 

\begin{theorem} 
   \label{main} 
   Let $K$ be a neighbourly triangulation of $S^{2n+1}$ on the vertex set $[m]=\{1,\ldots,m\}$. 
   Then $\zk$ is homotopy equivalent to a connected sum of products of two spheres. 
\end{theorem} 

If $K$ is a neighbourly non-polytopal triangulation of $S^{2n+1}$, it is not known in general 
whether~$\zk$ has a smooth structure. So it is natural to only expect that Theorem~\ref{main} 
might be upgraded to a homeomorphism. A key ingredient in our method is the homotopy 
type of a punctured copy of~$\zk$; while this is homotopy equivalent to the same wedge of 
spheres as the corresponding punctured connected sum, it is not clear whether this homotopy equivalence 
allows for the sort of embedding that would be needed to improve our result to a homeomorphism. 

If $K$ is a neighbourly triangulation of $S^{2n}$ then the homotopy type of the punctured 
copy of $\zk$ is unknown in general, and this currently blocks further progress. 

As our methods are purely homotopy theoretic, they can be extended. A generalized moment-angle 
complex $(D^{k},S^{k-1})^{K}$ for $k\geq 2$ is also a manifold if $K$ is a triangulation of a sphere.  
The authors are unaware whether the diffeomorphism type or homotopy type of a generalized 
moment-angle manifold is known even in the simplest case when $K$ is the boundary of an $m$-gon. 
We address this by proving the following. 

\begin{theorem} 
   \label{maingmam} 
   Let $K$ be a neighbourly triangulation of $S^{2n+1}$ on the vertex set $[m]$. If $k\geq 2$ and  
   then $(D^{k},S^{k-1})^{K}$ is homotopy equivalent to a connected sum of products of two spheres. 
\end{theorem}

\section{Moment-angle complexes associated to neighbourly triangulations} 
\label{sec:neighbourly} 

This section discusses properties of $\zk$ when $K$ is a neighbourly triangulation. Some 
statements are known while others are new. 

There is a general suspension splitting of moment-angle complexes due 
to Bahri, Bendersky, Cohen and Gitler~\cite[Corollary 2.23]{BBCG1}, known as the BBCG decomposition. 
For a simplicial complex $K$ on the vertex set $[m]$, if $I\subseteq [m]$ then 
the \emph{full subcomplex} $K_{I}$ of~$K$ is the simplicial complex consisting of those 
simplices in $K$ whose vertices are all in $I$. 

\begin{theorem}
    \label{BBCGsplit}
    Let $K$ be a simplicial complex. There is a homotopy equivalence 
    \[\Sigma \zk \simeq \bigvee\limits_{I \notin K}\Sigma^{2+|I|}|K_I|\] 
    that is natural for inclusions of simplicial complexes. $\qqed$
\end{theorem}

The simplicial complexes of interest are triangulations of spheres. Let $K$ be a 
triangulation of $S^{n}$ on $m$ vertices. Then by~\cite[Theorem 4.1.4]{BP}, 
$\zk$ is a $2$-connected manifold of dimension $n+m+1$. Since $\zk$ is simply-connected 
it is orientable and therefore satisfies Poincar\'{e} duality, implying that 
$H_{n+m+1}(\zk)\cong\mathbb{Z}$. Simple-connectivity then implies that $\zk$ can be 
given a $CW$-structure with a single cell in dimension $n+m+1$. 
Let $\overline{\zk}$ be the $(n+m)$-skeleton of $\zk$. There is a homotopy cofibration 
\[\nameddright{S^{n+m}}{f}{\overline{\zk}}{}{\zk}\] 
where $f$ attaches the $(n+m+1)$-cell to $\zk$. The BBCG decomposition of $\Sigma\zk$ 
leads to a related decomposition of $\Sigma\overline{\zk}$. 

\begin{lemma} 
   \label{BBCGbarzk} 
   Let $K$ be a triangulation of $S^{n}$ on $m$ vertices. Then there is a 
   homotopy equivalence  
   \[\Sigma\overline{\zk}\simeq \bigvee\limits_{I \notin K, I\neq [m]}\Sigma^{2+|I|}|K_I|.\] 
\end{lemma} 

\begin{proof} 
By Theorem~\ref{BBCGsplit}, $\Sigma \zk \simeq \bigvee\limits_{I \notin K}\Sigma^{2+|I|}|K_I|$. 
Observe that $I=[m]$ is not a face of $K$; for if it is then $K=\Delta^{m-1}$, implying that 
$\vert K\vert$ is contractible, and therefore $K$ is not a triangulation of~$S^{n}$. With $I=[m]$ 
we have $K_{[m]}=K$, so $\vert K_{[m]}\vert=\vert K\vert\simeq S^{n}$. Thus 
$\Sigma^{2+\vert [m]\vert}\vert K_{[m]}\vert\simeq S^{n+m+2}$. This implies that 
the $I=[m]$ summand accounts precisely for the $(n+m+2)$-cell of $\Sigma\zk$. Removing 
the $I=[m]$ summand removes the $(n+m+2)$-cell, giving 
$\Sigma\overline{\zk}\simeq \bigvee\limits_{I \notin K, I\neq [m]}\Sigma^{2+|I|}|K_I|$. 
\end{proof} 

The BBCG decomposition for $\Sigma \zk$ ``desuspends" if there is a homotopy equivalence 
\[\zk \simeq \bigvee\limits_{I \notin K}\Sigma^{1+|I|}|K_I|.\] 
The relevant instance of this in our case is from~\cite[Theorem 6.2]{ST}, which is 
stated in terms of pseudomanifolds of dimension $2n+1$, but we restrict to the special case of 
neighbourly triangulations of $S^{2n+1}$. Let $K\backslash i$ be the full subcomplex of $K$ 
on the vertex set $[m]-\{i\}$. 

\begin{theorem} 
   \label{ZK/i} 
   Let $K$ be a neighbourly triangulation of $S^{2n+1}$ on the vertex set $[m]$. Then the 
   BBCG decomposition for $\Sigma\mathcal{Z}_{K\backslash i}$ desuspends for all $i\in [m]$.~$\qqed$ 
\end{theorem} 

Theorem~\ref{ZK/i} is used to show that in the neighbourly case, the decomposition of 
$\Sigma\overline{\zk}$ in Lemma~\ref{BBCGbarzk} desuspends. The proof is a slightly modified 
version of~\cite[Proposition 6.5]{ST}; the details are included in full as they will be adapted later 
to the case of generalized moment-angle manifolds. 

\begin{proposition} 
   \label{partialdesuspend} 
   Let $K$ be a neighbourly triangulation of $S^{2n+1}$ on the vertex set $[m]$. Then there 
   is a homotopy equivalence  
   \[\overline{\zk}\simeq \bigvee\limits_{I \notin K, I\neq [m]}\Sigma^{1+|I|}|K_I|.\] 
   Further, for any $i\in [m]$ the map 
   \(\namedright{\mathcal{Z}_{K\backslash i}}{}{\zk}\) 
   induced by the inclusion 
   \(\namedright{K\backslash i}{}{K}\) 
   factors as a composite 
   \(\nameddright{\mathcal{Z}_{K\backslash i}}{}{\overline{\zk}}{}{\zk}\). 
\end{proposition} 

\begin{proof} 
Let $i\in [m]$ and consider the full subcomplex $K\backslash i$ of $K$. The simplicial inclusion 
\(\namedright{K\backslash i}{}{K}\) 
induces a map of moment-angle complexes 
\(\namedright{\mathcal{Z}_{K\backslash i}}{\iota}{\zk}\). 
Since the full subcomplexes of $K\backslash i$ are those for $K$ but without the vertex $i$, 
by Theorem~\ref{BBCGsplit} there are homotopy equivalences 
\[\Sigma\mathcal{Z}_{K\backslash i}\simeq\bigvee\limits_{I \notin K, i\notin I}\Sigma^{2+|I|}|K_I|\qquad 
   \Sigma\zk\simeq\bigvee\limits_{I \notin K}\Sigma^{2+|I|}|K_I|\] 
and the naturality statement in Theorem~\ref{BBCGsplit} implies that $\Sigma\iota$ is 
homotopic to an inclusion of wedge summands. In particular, the condition $i\notin I$ 
implies that $I=[m]$ is not one of the indices in the wedge decomposition for 
$\Sigma\mathcal{Z}_{K\backslash i}$, so Lemma~\ref{BBCGbarzk} implies that $\Sigma\iota$ 
factors as a composite 
\(\nameddright{\Sigma\mathcal{Z}_{K\backslash i}}{}{\Sigma\overline{\zk}}{}{\Sigma\zk}\). 
By Theorem~\ref{ZK/i}, the BBCG decomposition 
for $\mathcal{Z}_{K\backslash i}$ desuspends, giving a composite  
\[\nameddright{\bigvee\limits_{I \notin K, i\notin I}\Sigma^{1+|I|}|K_I|}{}{\overline{\zk}}{}{\zk}\] 
that induces an inclusion in homology. 
As this is true for each $i\in [m]$, taking the wedge sum over all $i\in [m]$ gives a map 
\[\namedright{\bigvee_{i=1}^{m}\bigg(\bigvee\limits_{I \notin K, i\notin I}\Sigma^{1+|I|}|K_I|\bigg)}{} 
    {\overline{\zk}}.\] 
Observe that this map induces an epimorphism in homology since its suspension accounts for all 
the summands in the decomposition of $\Sigma\overline{\zk}$. Note that the index set on the left 
may include duplicate copies of the same wedge summands. Restricting to a single copy for each 
instance of $I\notin K$, $I\neq [m]$ gives a map 
\[g\colon\namedright{\bigvee\limits_{I \notin K, i\notin I}\Sigma^{1+|I|}|K_I|}{}{\overline{\zk}}\] 
whose suspension induces the inclusion of the wedge summands in the decomposition of 
$\Sigma\overline{\zk}$ in Lemma~\ref{BBCGbarzk}. Thus $g$ induces an isomorphism in 
homology and so is a homotopy equivalence by Whitehead's Theorem. 
\end{proof} 

More is true, as proved in~\cite[Theorem 6.6]{ST}. 

\begin{theorem} 
   \label{neighbourlyinW} 
   Let $K$ be a neighbourly triangulation of $S^{2n+1}$ on the vertex set $[m]$. 
   Then each $\vert K_{I}\vert$ for $I\notin K$, $I\neq [m]$ in the homotopy equivalence for 
   $\overline{\zk}$ in Proposition~\ref{partialdesuspend} is either contractible or homotopy 
   equivalent to $\bigvee_{i=1}^{t_{I}} S^{n}$ for some finite positive integer $t_{I}$.~$\qqed$ 
\end{theorem} 

Proposition~\ref{partialdesuspend} and Theorem~\ref{neighbourlyinW} combine to give a 
homotopy equivalence for $\overline{\zk}$ that relates directly to the combinatorics of $K$. 

\begin{corollary} 
   \label{neighbourlysupport} 
   Let $K$ be a neighbourly triangulation of $S^{2n+1}$ on the vertex set $[m]$. Then 
   $\overline{\zk}\simeq\bigvee_{k=1}^{\ell} S^{n_{k}}$ where each $S^{n_{k}}$ retracts 
   off a $\mathcal{Z}_{K_{I_{k}}}$ for some $I_{k}$ with $\vert I_{k}\vert= n_{k}-n-1$. 
\end{corollary} 

\begin{proof} 
By Proposition~\ref{partialdesuspend} there is a homotopy equivalence 
$\overline{\zk}\simeq\bigvee_{I\notin K, I\neq [m]}\Sigma^{1+\vert I\vert}\vert K_{I}\vert$ 
and by Theorem~\ref{neighbourlyinW} each $\vert K_{I}\vert$ is either contractible or 
homotopy equivalent to a wedge $\bigvee_{i=1}^{t_{I}} S^{n}$. Thus, if~$\mathcal{I}$ is 
the subset of those $I\notin K$, $I\neq [m]$ with the property that $\vert K_{I}\vert$ is not 
contractible, then 
\[\overline{\zk}\simeq\bigvee_{I\in\mathcal{I}}\Sigma^{1+\vert I\vert}\big(\bigvee_{i=1}^{t_{I}} S^{n}\big) 
     \simeq\bigvee_{I\in\mathcal{I}}\big(\bigvee_{i=1}^{t_{I}} S^{n+1+\vert I\vert}\big).\] 

Re-indexing to suppress the role of $I$, there is a homotopy equivalence 
$\overline{\zk}\simeq\bigvee_{k=1}^{\ell} S^{n_{k}}$ where each $S^{n_{k}}$ retracts 
off a $\mathcal{Z}_{K_{I_{k}}}$ for some $I_{k}$ with $n_{k}=n+1+\vert I_{k}\vert$. 
\end{proof} 

Finally, a description is given of the integral cohomology of $\zk$ when $K$ is neighbourly. This 
can likely be deduced from a more general result in~\cite[Theorem 4.3]{KP}, while the  
rational cohomology was determined in~\cite[Proof of Theorem 5.5]{AB}. A proof is included because 
we want to transparently generalize it in Section~\ref{sec:gmam} to the case of generalized 
moment-angle complexes. 

\begin{proposition} 
   \label{neighbourlycohlgy} 
   Let $K$ be a neighbourly triangulation of $S^{2n+1}$. Then there is a ring isomorphism 
   $H^{\ast}(\zk)\cong H^{\ast}(M)$, where $M$ is a connected sum of products of two spheres. 
\end{proposition} 

\begin{proof}  
By Proposition~\ref{partialdesuspend} and Theorem~\ref{neighbourlyinW}, 
$\overline{\zk}\simeq \bigvee \limits_{I \notin K,I \neq [m]} \Sigma^{1+|I|}\vert K_I\vert$ 
is homotopy equivalent to a wedge of spheres. This implies two things. First, 
the only possible nonzero cup products occur in $H^{2n+m+2}(\zk)$. By Lemma~\ref{BBCGbarzk}, 
this cohomology group corresponds to the $K_{[m]}$ term in the BBCG decomposition 
of $\Sigma\zk$. On the other hand, in general, by~\cite[Theorem 4.5.8]{BP} nonzero cup products in
$H^{\ast}(\zk)$ occur only in cases of the form 
\(\namedright{H^{k-\vert I\vert +1}(K_{I})\otimes H^{\ell-\vert J\vert -1}(K_{J})}{} 
    {H^{k+\ell-\vert I\vert-\vert J\vert -1}(K_{I\cup J})}\) 
for $I\cap J=\emptyset$. Therefore, in our case, as the only nonzero cup product occurs with 
respect to $K_{[m]}$, we must also have $I\cup J=[m]$, that is, nonzero cup products occur 
only with respect to pairs of full subcomplexes $K_{I}$ and $K_{[m]-I}$. Second, 
$\overline{\zk}$ being homotopy equivalent to a wedge of spheres implies that $H_{\ast}(\zk)$ 
is torsion free. It has already been noted that $\zk$ is a $2$-connected manifold, implying that  
Poincar\'{e} duality holds, and as $H^{\ast}(\zk)$ is torsion free 
the cap and cup products are dual. Therefore the cup product induces 
isomorphisms $H^{t}(\zk)\cong H^{2n+m-t}(\zk)$ for all $0\leq t\leq 2n+m+2$. But as 
the only nonzero cup products occur with respect to pairs of full subcomplexes $K_{I}$ 
and $K_{[m]-I}$, the cup product must induce isomorphisms 
\begin{equation} 
  \label{bigradediso} 
  H^{t}(\Sigma^{1+\vert I\vert}\vert K_{I}\vert)\cong 
    H^{2n+m+2-t}(\Sigma^{1+\vert [m]-I\vert}\vert K_{[m]-I}\vert) 
\end{equation}   
for all $0\leq t\leq 2n+m+2$. 

Fix $I\notin K$, $I\neq [m]$. By Theorem~\ref{neighbourlyinW}, $\Sigma^{1+\vert I\vert}\vert K_{I}\vert$ 
is homotopy equivalent to a finite wedge of spheres of the same dimension, say 
$\Sigma^{1+\vert I\vert}\vert K_{I}\vert\simeq\bigvee_{j=1}^{r} S^{d}$. Theorem~\ref{neighbourlyinW} 
also implies that $\Sigma^{1+\vert [m]-I\vert}\vert K_{[m]-I}\vert$ is homotopy equivalent to a finite 
wedge of spheres, so the isomorphism~(\ref{bigradediso}) implies that 
$\Sigma^{1+\vert [m]-I\vert}\vert K_{[m]-I}\vert\simeq\bigvee_{j=1}^{r} S^{2n+m+2-d}$. 
Let $\{x_{1},\ldots, x_{r}\}$ be 
a basis of $H^{d}(\bigvee_{j=1}^{d} S^{d})$ corresponding to the generators of the spheres 
in the wedge. Then the cup product isomorphism from Poincar\'{e} duality implies that 
$\{\overline{x}_{1},\ldots, \overline{x}_{r}\}$ is a basis of 
$H^{2n+m+2-d}(\bigvee_{j=1}^{r} S^{2n+m+2-d})$, 
where $x_{i}\cup\overline{x}_{j}$ equals the generator of $H^{2n+m+2}(\zk)$ if $i=j$ and equals 
$0$ otherwise. This basis may not correspond to the basis given by the inclusion of the 
wedge summands. In that case, there is a linear isomorphism between bases and this 
can be used to construct a self-equivalence of $\bigvee_{j=1}^{r} S^{2n+m+2-d}$. In this 
way, the wedge $\bigvee_{j=1}^{r} S^{2n+m+2-d}$ can be changed by a homotopy 
equivalence so that the basis $\{\overline{x}_{1},\ldots,\overline{x}_{r}\}$ does correspond 
to the inclusion of the wedge summands. 
This basis has the property that the cup product pairs each wedge summand in 
$\bigvee_{j=1}^{r} S^{d}$ with one wedge summand in $\bigvee_{j=1}^{r} S^{2n+m+2-d}$ so that 
$H^{\ast}(\conn_{j=1}^{r} S^{d}\times S^{2n+m+2-d})$ is a subalgebra of $H^{\ast}(\zk)$. 
This can be done for each pair $I$ and $[m]-I$ where $I\notin K$, $I\neq [m]$. The one 
remaining pair is $[m]$ and $\emptyset$, which corresponds to the $2n+m+2$ and $0$ dimensional 
cohomology classes in $H^{\ast}(\zk)$. Thus, collectively, there is a ring isomorphism 
$H^{\ast}(\zk)\cong H^{\ast}(M)$ where $M$ is a connected sum of products of two spheres. 
\end{proof}

\section{An alternative homotopy equivalence for $\overline{\zk}$}  

By Proposition~\ref{partialdesuspend}, if $K$ is a neighbourly triangulation of $S^{2n+1}$ on $[m]$ 
then there is a homotopy equivalence 
$\overline{Z_{K}}\simeq\bigvee_{I \notin K,I \neq [m]} \Sigma^{1+\vert I\vert}\vert K_I\vert$ 
where each $\vert K_{I}\vert$ is contractible or homotopy equivalent to a finite wedge of spheres. 
It will be helpful in what follows to construct an alternative homotopy equivalence for $\overline{\zk}$ 
that has particular properties. 

Fix an $I\notin K$, $I\neq [m]$. Consider the simplicial inclusion 
\(\namedright{K_{I}}{}{K}\). 
As $I$ is a proper subset of~$[m]$ there is a vertex $i\notin I$, implying that the simplicial 
inclusion 
\(\namedright{K_{I}}{}{K}\) 
factors as a composite of simplicial inclusions 
\(\nameddright{K_{I}}{}{K\backslash i}{}{K}\).  
This composite induces maps of polyhedral products 
\(\nameddright{\mathcal{Z}_{K_{I}}}{}{\mathcal{Z}_{K\backslash i}}{}{\zk}\).  
By Proposition~\ref{partialdesuspend}, the map 
\(\namedright{\mathcal{Z}_{K\backslash i}}{}{\zk}\) 
factors as a composite 
\(\nameddright{\mathcal{Z}_{K\backslash i}}{}{\overline{\zk}}{}{\zk}\). 

In general, if $J\subseteq [m]$ then the full subcomplex $K_{J}$ of $K$ has the property 
that the simplicial inclusion 
\(\namedright{K_{J}}{}{K}\) 
induces a map of polyhedral products 
\(\namedright{\mathcal{Z}_{K_{J}}}{}{\zk}\) 
that, by~\cite{DS}, has a left inverse 
\(\namedright{\zk}{}{\mathcal{Z}_{K_{J}}}\). 
In our case, as $K_{I}$ is a full subcomplex of $K\backslash i$ the map 
\(\namedright{\mathcal{Z}_{K_{I}}}{}{\mathcal{Z}_{K\backslash i}}\) 
has a left inverse 
\(\namedright{\mathcal{Z}_{K\backslash i}}{}{\mathcal{Z}_{K_{I}}}\), 
and as $K\backslash i$ is a full subcomplex of $K$ the map 
\(\namedright{\mathcal{Z}_{K\backslash i}}{}{\zk}\) 
has a left inverse 
\(\namedright{\zk}{}{\mathcal{Z}_{K\backslash i}}\). 
Thus there is a sequence of maps 
\begin{equation} 
  \label{Zstring} 
  \nameddright{\mathcal{Z}_{K_{I}}}{}{\mathcal{Z}_{K\backslash i}}{}{\overline{\zk}} 
    \longrightarrow\nameddright{\zk}{}{\mathcal{Z}_{K\backslash i}}{}{\mathcal{Z}_{K_{I}}} 
\end{equation}  
that equals the identity map. We obtain the following. 

\begin{lemma} 
    \label{ZKIproperties} 
    Let $K$ be a neighbourly triangulation of $S^{2n+1}$ on $[m]$. If $I\notin K$, $I\neq [m]$ then: 
    \begin{itemize} 
       \item[(a)] $\mathcal{Z}_{K_{I}}$ is a retract of $\overline{\zk}$; 
       \item[(b)] $\mathcal{Z}_{K_{I}}$ is contractible or homotopy equivalent to a finite wedge of spheres; 
       \item[(c)] the BBCG decomposition for $\Sigma\mathcal{Z}_{K_{I}}$ desuspends. 
    \end{itemize} 
\end{lemma} 

\begin{proof} 
Part~(a) immediately follows from the fact that the composite in~(\ref{Zstring}) is the identity map. 
By Theorem~\ref{neighbourlyinW}, $\overline{\zk}$ is contractible or homotopy equivalent to a finite wedge 
of spheres. Any retract of a contractible space is contractible, and any retract of a wedge of spheres is 
homotopy equivalent to a subwedge of those spheres, so part~(b) follows. As $\mathcal{Z}_{K_{I}}$ is 
contractible or homotopy equivalent to a wedge of spheres, its suspension will give a homotopy 
equivalence for $\Sigma\mathcal{Z}_{K_{I}}$ that is contractible or a wedge of spheres, which 
implies that each of the wedge summands in the BBCG decomposition of $\Sigma\mathcal{Z}_{K_{I}}$ 
is contractible or homotopy equivalent to a wedge of spheres. Thus each such wedge summand 
desuspends, implying the BBCG decomposition desuspends, proving part~(c). 
\end{proof} 

The desuspension of the BBCG decomposition for $\Sigma\mathcal{Z}_{K_{I}}$ implies 
by Theorem~\ref{BBCGsplit} that $\Sigma^{1+\vert I\vert}\vert K_{I}\vert$ is a retract 
of $\mathcal{Z}_{K_{I}}$. Thus there is a composite  
\(\nameddright{\Sigma^{1+\vert I\vert}\vert K_{I}\vert}{}{\mathcal{Z}_{K_{I}}} 
    {}{\Sigma^{1+\vert I\vert}\vert K_{I}\vert}\) 
that is homotopic to the identity map. Define $a_{I}$ and $p_{I}$ by the composites 
\[a_{I}\colon\nameddright{\Sigma^{1+\vert I\vert}\vert K_{I}\vert}{}{\mathcal{Z}_{K_{I}}}{}{\overline{\zk}}\]  
\[p_{I}\colon\namedddright{\overline{\zk}}{}{\zk}{}{\mathcal{Z}_{K_{I}}}{}{\Sigma^{1+\vert I\vert}\vert K_{I}\vert}\] 
where the right map in the composite for $a_{I}$ and the left and middle maps in the 
composite for $p_{I}$ are as in~(\ref{Zstring}). Then $p_{I}\circ a_{I}$ is homotopic to the 
identity map. Let 
\[\overline{a}\colon\namedright{\bigvee_{I\notin K,I\neq [m]}\Sigma^{1+\vert I\vert}\vert K_{I}\vert} 
     {}{\overline{\zk}}\] 
be the wedge sum of the maps $a_{I}$. Since $\overline{\zk}$ is contractible or homotopy equivalent 
to a wedge of spheres, it is a co-$H$-space. Thus the maps $p_{I}$ can be added to give a map 
\[\overline{p}\colon\namedright{\overline{\zk}}{}
    {\hspace{-3mm}\bigvee_{I\notin K, I\neq [m]}\Sigma^{1+\vert I\vert}\vert K_{I}\vert}.\] 

\begin{proposition} 
   \label{ZKequivs} 
   Let $K$ be a neighbourly triangulation of $S^{2n+1}$ on $[m]$. Then the maps $\overline{a}$ 
   and $\overline{p}$ are both homotopy equivalences. 
\end{proposition} 

\begin{proof}  
Fix $J\notin K$, $J\neq [m]$. Since $p_{J}\circ a_{J}$ is homotopic to the identity map, the composite 
\[\nameddright{\Sigma^{1+\vert J\vert}\vert K_{J}\vert}{a_{J}}{\overline{\zk}} 
      {\overline{p}}{\hspace{-3mm}\bigvee_{I\notin K, I\neq [m]}\Sigma^{1+\vert I\vert}\vert K_{I}\vert}\] 
surjects in homology onto the submodule $H_{\ast}(\Sigma^{1+\vert J\vert}\vert K_{J}\vert)$ 
of $H_{\ast}(\bigvee_{I\notin K, I\neq [m]}\Sigma^{1+\vert I\vert}\vert K_{I}\vert)$. This is 
true for every $J$, so as $\overline{a}$ is a wedge sum of the maps $a_{J}$, the composite 
$\overline{p}\circ\overline{a}$ induces a surjection in homology. As the domain and range of 
$\overline{p}\circ\overline{a}$ is the same space and it is of finite type, it follows that 
$\overline{p}\circ\overline{a}$ induces an isomorphism in homology. As spaces are simply-connected, 
Whitehead's Theorem then implies that $\overline{p}\circ\overline{a}$ is a homotopy equivalence. 

In particular, this implies that $\bigvee_{I\notin K, I\neq [m]}\Sigma^{1+\vert I\vert}\vert K_{I}\vert$ 
is a retract of $\overline{\zk}$. But by Proposition~\ref{partialdesuspend}, $\overline{\zk}$ has the same 
homotopy type as $\bigvee_{I\notin K, I\neq [m]}\Sigma^{1+\vert I\vert}\vert K_{I}\vert$. Thus 
$\overline{a}$ and $\overline{p}$ are both homotopy equivalences. 
\end{proof} 

The homotopy equivalence $\overline{p}$ will be the more useful one in what follows.

\section{One more ingredient} 

This brief section discusses an additional result by Stasheff~\cite{S} that will play a role in the 
proof of Theorem~\ref{main}. 

\begin{theorem} 
   \label{Stasheff} 
   Let $N$ be a simply-connected $n$ dimensional Poincar\'{e} Duality complex and 
   let $\overline{N}$ be the $(n-1)$-skeleton of $N$.  
   If $Y=\overline{N}\cup e^{n}$ and there is a ring isomorphism 
   $H^{\ast}(Y;\mathbb{Q})\cong H^{\ast}(N;\mathbb{Q})$ 
   then there is a rational homotopy equivalence $Y\simeq N$.~$\qqed$  
\end{theorem} 

\begin{example} 
  \label{Stasheffexample} 
  Let $K$ be a neighbourly triangulation of $S^{2n+1}$ on $m$ vertices. Then $\zk$ is a 
  simply-connected $(m+2n+2)$-dimensional Poincar\'{e} Duality complex. By 
  Proposition~\ref{neighbourlycohlgy}, there is a ring isomorphism 
  
  $H^{\ast}(\zk;\mathbb{Q})\cong H^{\ast}(M;\mathbb{Q})$ 
  where $M$ is a connected sum of products of two spheres. By Proposition~\ref{partialdesuspend}, 
  $\overline{\zk}$ is homotopy equivalent to a wedge of spheres, 
  say $\overline{\zk}\simeq\bigvee_{i=1}^{r} S^{n_{i}}$. Since $M$ is a connected sum of products 
  of two spheres, $\overline{M}$ is also homotopy equivalent to a wedge of spheres. The 
  isomorphism $H^{\ast}(\zk)\cong H^{\ast}(M)$ implies that 
  $H^{\ast}(\overline{\zk})\cong H^{\ast}(\overline{M})$, and therefore 
  $\overline{M}\simeq\bigvee_{i=1}^{r} S^{n_{i}}$.  Thus if $X=\bigvee_{i=1}^{r} S^{n_{i}}$ then 
  $\zk=X\cup_{f} e^{2n+m+2}$ and $M=X\cup_{g} e^{2n+m+1}$ for maps $f$ and $g$ that attach 
  the $(m+2n+2)$-cell in each case. Theorem~\ref{Stasheff} then implies that there is a rational 
  homotopy equivalence $\zk\simeq M$. 
\end{example}

\section{The proof of Theorem~\ref{main}} 
\label{sec:proof} 

Let $K$ be a triangulation of $S^{2n+1}$ on the vertex set $[m]$. There is a homotopy 
cofibration 
\[\nameddright{S^{2n+m+1}}{f}{\overline{\zk}}{}{\zk}\] 
where $f$ attaches the $(2n+m+2)$-cell to $\zk$. 
By Proposition~\ref{ZKequivs}, there is a homotopy equivalence 
\(\namedright{\overline{\zk}}{\overline{p}}{\bigvee_{I\notin K, I\neq [m]}\Sigma^{1+\vert I\vert}\vert K_{I}\vert}\). 

\begin{lemma} 
   \label{fscnull} 
   For each $I\notin K$, $I\neq [m]$, the composite 
   \(\nameddright{S^{2n+m+1}}{f}{\overline{\zk}}{p_{I}}{\Sigma^{1+\vert I\vert}\vert K_{I}\vert}\) 
   is null homotopic. 
\end{lemma} 

\begin{proof} 
Fix an $I\notin K$, $I\neq [m]$. By definition, $p_{I}$ factors as a composite 
\(\nameddright{\overline{\zk}}{}{\zk}{}{\Sigma^{1+\vert I\vert}\vert K_{I}\vert}\) 
where the left map is the skeletal inclusion. Thus $p_{I}\circ f$ factors through the composite 
\(\nameddright{S^{2n+m+1}}{f}{\overline{\zk}}{}{\zk}\), 
which is null homotopic since it is two consecutive maps in a homotopy cofibration. Therefore 
$p_{I}\circ f$ is null homotopic. 
\end{proof} 

Let $p$ be the composite 
\[p\colon\nameddright{\overline{\zk}}{\overline{p}} 
     {\hspace{-3mm}\bigvee_{I\notin K,I\neq [m]}\Sigma^{1+\vert I\vert}\vert K\vert} 
     {}{\hspace{-3mm}\prod_{I\notin K,I\neq [m]}\Sigma^{1+\vert I\vert}\vert K\vert}\]  
where the right map is the inclusion of the wedge into the product. By definition, 
$\overline{p}$ is the sum of the maps $p_{I}$, so $p$ is the product of the maps $p_{I}$. 
By Lemma~\ref{fscnull}, each composite $p_{I}\circ f$ is null homotopic, implying that $p\circ f$ 
is null homotopic. 

It will be more notationally convenient if $\overline{\zk}$ and $\overline{p}$ are written in 
terms of wedges of spheres. By Theorem~\ref{neighbourlyinW}, each $\vert K_{I}\vert$ is 
contractible or homotopy equivalent to a finite wedge of spheres. Thus 
$\Sigma^{1+\vert I\vert}\vert K_{I}\vert$ is contractible or homotopy equivalent to a finite wedge 
of spheres. Taking the wedge sum of these homotopy equivalences for $I\notin K$, $I\neq [m]$ 
gives a composite of homotopy equivalences 
\[\overline{q}\colon\nameddright{\overline{\zk}}{\overline{p}} 
      {\hspace{-3mm}\bigvee_{I\notin K,I\neq [m]}\Sigma^{1+\vert I\vert}\vert K\vert}{\simeq} 
      {\bigvee_{k=1}^{\ell} S^{n_{k}}}\] 
for some $\ell\geq 1$. In what follows, the homotopy equivalence $\overline{q}$ will be 
suppressed and $\overline{\zk}$ will simply be regarded as being the wedge of spheres 
$\bigvee_{k=1}^{\ell} S^{n_{k}}$. Let 
\[q\colon\namedright{\overline{\zk}} {}{\prod_{k=1}^{\ell} S^{n_{k}}}\]  
be the inclusion of the wedge into the product. Then $q$ is a product of maps, 
each of which factors through $p_{I}$ for some $I$, implying that each composes trivially with $f$. 
Thus $q\circ f$ is null homotopic. 

The null homotopy for $q\circ f$ implies that $q$ extends across the skeletal inclusion 
\(\namedright{\overline{\zk}}{}{\zk}\) 
to a map 
\(q'\colon\namedright{\zk}{}{\prod_{k=1}^{\ell} S^{n_{k}}}\). 
Define the spaces $E$, $F$ and $G$, and the maps $g$ and $\phi$ by the homotopy pullback diagram 
\begin{equation} 
  \label{EFGpb} 
  \diagram 
     E\rto\ddouble & F\rto^-{g}\dto^{\phi} & G\dto \\ 
     E\rto & \overline{\zk}\rto\dto^{q} & \zk\dto^{q'} \\ 
     & \prod_{k=1}^{\ell} S^{n_{k}}\rdouble & \prod_{k=1}^{\ell} S^{n_{k}}. 
  \enddiagram 
\end{equation}  
Since there is a homotopy cofibration 
\(\nameddright{S^{2n+m+1}}{f}{\overline{\zk}}{}{\zk}\) 
and spaces are simply-connected, the Serre exact sequence or the Blakers-Massey Theorem 
implies that $E$ is $(2n+m)$-connected, $H^{2n+m+1}(E)\cong\mathbb{Z}$, and $f$ can be 
written as the composite 
\(\nameddright{S^{2n+m+1}}{}{E}{}{\overline{\zk}}\) 
where the left map is the inclusion of the bottom cell. Let $\lambda$ be the composite 
\[\lambda\colon\nameddright{S^{2n+m+1}}{}{E}{}{F}.\] 
Then $\lambda$ is a lift of $f$ to $F$, and has the property that $g\circ\lambda$ is null 
homotopic. 

It will be easier in what follows to work with adjoints. Let $\widetilde{f}$ and $\widetilde{\lambda}$ 
be the adjoints of $f$ and $\lambda$ respectively. Then taking adjoints gives a homotopy 
commutative diagram 
\begin{equation} 
  \label{Fadjlift} 
  \diagram 
      S^{2n+m}\rto^-{\widetilde{\lambda}}\drto_{\widetilde{f}} & \Omega F\dto^{\Omega\phi} \\ 
      & \Omega\overline{\zk}.  
  \enddiagram 
\end{equation}

%
%
%

The homotopy type of $\Omega F$ and the homotopy class of $\Omega\phi$ are given by the 
Hilton-Milnor Theorem. To be more explicit, some notation is necessary. For $1\leq j\leq\ell$, let 
\[\nu_{j}\colon\namedright{S^{n_{j}}}{}{\bigvee_{k=1}^{\ell} S^{n_{k}}}\] 
be the inclusion of the $j^{th}$-wedge summand. Let $L\langle v_{1},\ldots,v_{\ell}\rangle$ be the 
free graded Lie algebra generated by elements $v_{1},\ldots,v_{\ell}$ with $\vert v_{j}\vert = n_{i}-1$. 
Let $\mathcal{B}$ be a basis for the submodule of commutators of length~$\geq 2$ in 
$L\langle v_{1},\ldots,v_{\ell}\rangle$. Each $\beta\in\mathcal{B}$ corresponds to an iterated 
commutator of the elements $v_{1},\ldots,v_{\ell}$. Let $d(\beta)$ be the degree of $\beta$. Let 
\[w_{\beta}\colon\namedright{S^{d(\beta)+1}}{}{\bigvee_{k=1}^{\ell} S^{n_{k}}}\] 
be the iterated Whitehead product corresponding to $\beta$, where each instance of $v_{j}$ 
in the bracket $\beta$ is replaced by the map $\nu_{j}$ in the Whitehead product $w_{\beta}$. 

\begin{theorem}[Hilton-Milnor] 
   \label{HM} 
   There is a homotopy equivalence 
   $\Omega F\simeq\prod_{\beta\in\mathcal{B}}\Omega S^{d(\beta)+1}$ 
   under which the restriction of the map $\Omega\phi$ to $\Omega S^{d(\beta)+1}$ 
   is $\Omega w_{\beta}$.~$\qqed$ 
\end{theorem} 

\begin{proof}[Proof of Theorem~\ref{main}] 
We identify the homotopy class of the attaching map 
\(\namedright{S^{2n+m+1}}{f}{\bigvee_{k=1}^{\ell} S^{n_{k}}}\) 
for the top cell of $\zk$. It is equivalent to identify the homotopy class of the adjoint 
\(\namedright{S^{2n+m}}{\widetilde{f}}{\Omega\overline{\zk}}\) 
of $f$. By~(\ref{Fadjlift}), $\widetilde{f}$ factors as a composite 
\(\nameddright{S^{2n+m}}{\widetilde{\lambda}}{\Omega F}{\Omega\phi}{\Omega\overline{\zk}}\) 
and by Theorem~\ref{HM}, $\Omega\phi\circ\widetilde{\lambda}$ may be rewritten as a composite 
\[\nameddright{S^{2n+m}}{\kappa}{\prod_{\beta\in\mathcal{B}}\Omega S^{d(\beta)+1}}{W} 
    {\Omega\overline{\zk}}\] 
where $W=\prod_{\beta\in\mathcal{B}}\Omega w_{\beta}$. The factors in $W\circ\kappa$ 
are analyzed one-by-one. Define 
\[\kappa_{\beta}\colon\namedright{S^{2n+m}}{}{\Omega S^{d(\beta)+1}}\] 
by projecting $\kappa$ to the factor indexed by $\beta$. Then $W\circ\kappa$ 
is homotopic to the sum of the maps $\Omega w_{\beta}\circ\kappa_{\beta}$. 
There are three cases. 
\smallskip 

\noindent 
\textit{Case 1: $2n+m<d(\beta)$}. Then the connectivity of $\Omega S^{d(\beta)+1}$ is greater 
than or equal to the dimension of $S^{2n+m}$, implying that the map 
\(\namedright{S^{2n+m}}{\kappa_{\beta}}{\Omega S^{d(\beta)+1}}\) 
is null homotopic. Thus $\Omega w_{\beta}\circ\kappa_{\beta}$ is null homotopic. 
\smallskip 

\noindent 
\textit{Case 2: $2n+m=d(\beta)$}. Then the map 
\(\namedright{S^{2n+m}}{\kappa_{\beta}}{\Omega S^{d(\beta)+1}}\) 
is some multiple $t$ of the inclusion of the bottom cell. This inclusion is the suspension map 
\(\namedright{S^{2n+m}}{E}{\Omega S^{2n+m+1}}\),  
defined as the adjoint of the identity map on $S^{2n+m+1}$. Thus 
$\Omega w_{\beta}\circ\kappa_{\beta}\simeq t\cdot (\Omega w_{\beta}\circ E)$. 
Observe that the adjoint of $\Omega w_{\beta}\circ E$ is the iterated Whitehead product 
\(\namedright{S^{d(\beta)+1}}{w_{\beta}}{\overline{\zk}}\). 
Therefore, as $\Omega w_{\beta}\circ\kappa_{\beta}$ is a component in the homotopy 
class of $\widetilde{f}$, on taking adjoints we obtain $t\cdot w_{\beta}$ as a component in the 
homotopy class of $f$. 

By Example~\ref{Stasheffexample}, there is a rational homotopy equivalence 
$\zk\simeq M$, where $M$ is a connected sum of products of two spheres. Thus the attaching 
map $f$ for the top cell of $\zk$ is rationally homotopic to 
$\sum_{1\leq i,j\leq r} s_{i,j}\cdot [\nu_{i},\nu_{j}]$, 
where $s_{i,j}=1$ if $S^{n_{i}}\times S^{n_{j}}$ appears as one of the factors in the connected 
sum for $M$ and $s_{i,j}=0$ otherwise. 

The homotopy fibration 
\(\nameddright{F}{\phi}{\bigvee_{k=1}^{\ell} S^{n_{k}}}{}{\prod_{k=1}^{\ell} S^{n_{k}}}\) 
splits after looping to give a homotopy equivalence 
\[\Omega(\bigvee_{k=1}^{\ell} S^{n_{k}})\simeq\prod_{k=1}^{\ell} S^{n_{k}}\times\Omega F,\] 
implying that $\Omega\phi$ has a left homotopy inverse. By Theorem~\ref{HM}, 
the restriction of $\Omega\phi$ to $\Omega S^{d({\beta})+1}$ is $\Omega w_{\beta}$, implying 
that $\Omega w_{\beta}$ has a left homotopy inverse. In particular, $\Omega w_{\beta}$ is 
rationally nontrivial, implying that $w_{\beta}$ is rationally nontrivial. Thus as $t\cdot w_{\beta}$ 
is a component in the homotopy class of $f$, it is also a component in the rational homotopy 
class of $f$. Hence $t\neq 0$ if and only if $w_{\beta}=[\nu_{i},\nu_{j}]$ for some $i,j$ with 
$S^{n_{i}}\times S^{n_{j}}$ appearing as one of the factors in the connected sum $M$. 
Therefore $\Omega w_{\beta}\circ\kappa_{\beta}$ is non-trivial if and only if 
$w_{\beta}=[\nu_{i},\nu_{j}]$ for some $i,j$ with $S^{n_{i}}\times S^{n_{j}}$ appearing as 
one of the factors in the connected sum $M$. In the non-trivial case, the 
Whitehead product $[\nu_{i},\nu_{j}]$ is detected by the cup product in cohomology, so 
the value of $t$ can be determined. By Proposition~\ref{neighbourlycohlgy}, there is an integral 
ring isomorphism $H^{\ast}(\zk)\cong H^{\ast}(M)$, so $t=1$. Thus 
$\Omega w_{\beta}\circ\kappa_{\beta}\simeq [\nu_{i},\nu_{j}]$. 
\smallskip 

\noindent 
\textit{Case 3: $2n+m>d(\beta)$}. Then the map 
\(\namedright{S^{2n+m}}{\kappa_{\beta}}{\Omega S^{d(\beta)+1}}\) 
has an adjoint that represents a homotopy group of the form $\pi_{r+t}(S^{r})$ for some $t\geq 1$. 
It is well known that such a group is torsion unless $r$ is even and $t=r-1$. This results in 
two sub-cases. 
\smallskip 

\noindent 
\textit{Case 3(i): $d(\beta)+1$ is even, $2n+m=2d(\beta)$ and $\kappa_{\beta}$ is rationally 
nontrivial}. Then the adjoint  
\(\widetilde{\kappa}_{\beta}\colon\namedright{S^{2n+m+1}}{}{S^{d(\beta)+1}}\) 
of $\kappa_{\beta}$ satisfies $\widetilde{\kappa}_{\beta}\simeq t\cdot [\iota,\iota]$, where $\iota$ is the 
generator of $\pi_{d(\beta)+1}(S^{d(\beta)+1})$ and $t\neq 0$. The naturality of the Whitehead product then 
implies that $w_{\beta}\circ\widetilde{\kappa}_{\beta}$ is a Whitehead product of length twice 
the length of $w_{\beta}$. As the length of $w_{\beta}$ is at least $2$, the length of 
$w_{\beta}\circ\widetilde{\kappa}_{\beta}$ is at least $4$. This implies that the homotopy class 
of $f$ has a component involving a Whitehead product of length larger than $2$, and therefore 
so does the rational homotopy class of $f$, but we have already seen in Case 2 that this cannot 
occur. So it must be that $t=0$, a contradiction. Thus Case 3(i) does not occur. 
\smallskip 

\noindent 
\textit{Case 3(ii): $2n+m>d(\beta)$ and $\kappa_{\beta}$ is a torsion class}. We will show that 
$\kappa_{\beta}$ is null homotopic. This begins by comparing the number of vertices 
supporting~$w_{\beta}$ to $m$. The map $w_{\beta}$ is an iterated Whitehead product of the form 
$w_{\beta}=[\nu_{1},\ldots,[\nu_{r-1},\nu_{r}]]$ for some $r\geq 2$. As 
each $\nu_{i}$ is a map 
\(\namedright{S^{n_{i}}}{}{\overline{\zk}}\), 
the iterated Whitehead product 
\(\namedright{S^{d(\beta)+1}}{w_{\beta}}{\overline{\zk}}\) 
satisfies $d(\beta)+1=(\sum_{i=1}^{r}n_{i})-r+1$. By assumption, $d(\beta)<2n+m$, implying that 
\begin{equation} 
  \label{wbetadegree} 
  (\textstyle\sum_{i=1}^{r} n_{i})-r<2n+m. 
\end{equation}  
On the other hand, as $K$ is neighbourly, it is $n$-neighbourly, so Lemma~\ref{neighbourlysupport} 
implies that each~$S^{n_{i}}$ in the wedge sum for $\overline{\zk}$ retracts off a 
$\mathcal{Z}_{K_{I_{i}}}$ for some $I_{i}$ with $\vert I_{i}\vert= n_{i}-n-1$.  
Thus $w_{\beta}$ is supported by at most $\sum_{i=1}^{r} (n_{i}-n-1)=(\sum_{i=1}^{r} n_{i})-r(n+1)$  
vertices. Since $r\geq 2$, the largest number of vertices needed to support $w_{\beta}$ is 
\begin{equation} 
  \label{wbetavertices} 
  (\textstyle\sum_{i=1}^{r} n_{i})-2(n+1). 
\end{equation}  
Comparing~(\ref{wbetadegree}) and~(\ref{wbetavertices}) we obtain 
$(\sum_{i=1}^{r} n_{i})-2(n+1)<m+r-2$, and this holds for any $r\geq 2$. In particular, 
from $r=2$ we obtain 
\begin{equation} 
  \label{wbetafullsc} 
  \textstyle(\sum_{i=1}^{r} n_{i})-2(n+1)<m. 
\end{equation}  

Tying things together, $w_{\beta}$ is a Whitehead product on $S^{n_{1}}\vee\cdots\vee S^{n_{r}}$, 
which is a subwedge of $\overline{\zk}$. The vertex set supporting $w_{\beta}$ is the union 
of those for each $S^{n_{i}}$, which is $I_{1}\cup\cdots\cup I_{r}$. 
By~(\ref{wbetavertices}) and~(\ref{wbetafullsc}), the number of vertices supporting $w_{\beta}$ is 
strictly less than $m$. Thus $I_{1}\cup\cdots\cup I_{r}$ is a proper subset of~$[m]$. 
By Proposition~\ref{partialdesuspend}, the only full subcomplex of $\zk$ 
that is not a retract of $\overline{\zk}$ is $K_{I}$ for $I=[m]$. Therefore 
$\mathcal{Z}_{K_{I_{1}\cup\cdots\cup I_{r}}}$ is a retract of both $\overline{\zk}$ and $\zk$. 
In particular, as $\overline{\zk}$ is homotopy equivalent to a wedge of spheres, 
so is $\mathcal{Z}_{K_{I_{1}\cup\cdots\cup I_{r}}}$, and therefore 
$S^{n_{1}}\vee\cdots\vee S^{n_{r}}$ is a retract of $\mathcal{Z}_{K_{I_{1}\cup\cdots\cup I_{r}}}$.  
Thus there is a commutative diagram 
\begin{equation} 
  \label{wedgeretract} 
  \diagram 
      S^{n_{1}}\vee\cdots\vee S^{n_{r}}\rto\drdouble & \overline{\zk}\rto^-{j}\dto & \zk\dto \\ 
      & S^{n_{1}}\vee\cdots\vee S^{n_{r}}\rdouble & S^{n_{1}}\vee\cdots\vee S^{n_{r}}  
  \enddiagram 
\end{equation}   
where $j$ is the skeletal inclusion. 

In general, if $I\subseteq [m]$ then the retraction  
\(\nameddright{\mathcal{Z}_{K_{I}}}{}{\zk}{}{\mathcal{Z}_{K_{I}}}\) 
has the property that the right map is determined by a projection. By its definition, $\zk$ is 
a union of products of $m$ factors, each of which is a subspace of $D^{2}$. Projecting onto 
the factors in $I$ gives a composite 
\(\nameddright{\zk}{}{\prod_{i=1}^{m} D^{2}}{}{\prod_{j\in I} D^{2}}\) 
whose image is $\mathcal{Z}_{K_{I}}$. In our case, this means that the map 
\(\namedright{\zk}{}{\bigvee_{i=1}^{r} S^{n_{i}}}\) 
obtained as a retract induced by a full subcomplex is obtained through projection onto 
vertex coordinates. Thus when~(\ref{wedgeretract}) is combined with the homotopy 
pullback~(\ref{EFGpb}) defining $F$ and $G$ we obtain a homotopy fibration diagram 
\[\diagram 
      H\rto\dto & F\rto^-{g}\dto^{\phi} & G\rto\dto & H\dto \\ 
      \bigvee_{i=1}^{r} S^{n_{i}}\rto\dto^{a} & \overline{\zk}\rto^-{j}\dto^{q} & \zk\rto\dto^{q'} 
          & \bigvee_{i=1}^{r} S^{n_{i}}\dto^{a} \\ 
      \prod_{i=1}^{r} S^{n_{i}}\rto^-{b} & \prod_{k=1}^{\ell} S^{n_{k}}\rdouble 
          & \prod_{k=1}^{\ell} S^{n_{k}}\rto^-{c} & \prod_{i=1}^{r} S^{n_{i}} 
  \enddiagram\] 
where the columns are homotopy fibrations, $a$ is the inclusion of the wedge into the product 
and its homotopy fibre is defined to be $H$, $b$ is the inclusion, $c$ is a projection, 
and $c\circ b$ is the identity map. As the middle and bottom rows in this homotopy fibration 
diagram are identity maps, the top row is a homotopy equivalence. The naturality of the 
Hilton-Milnor decomposition with respect to inclusions and projections of wedges implies 
that the factors in the homotopy decomposition for $\Omega F$ that come from $\Omega H$ 
map through $\Omega g$ to retract off $\Omega G$. Consequently, as $\lambda$ factors through 
the homotopy fibre of $g$ and as $w_{\beta}$ is a Whitehead product on the subwedge 
$\bigvee_{i=1}^{r} S^{n_{i}}$ of $\bigvee_{k=1}^{\ell} S^{n_{k}}$, the composite 
\(\nameddright{S^{2n+m}}{\widetilde{\lambda}}{\Omega F}{\mbox{\tiny proj}}{\Omega S^{d(\beta)+1}}\) 
factors through $\Omega g$ and so is null homotopic. This composite is exactly the definition 
of $\kappa_{\beta}$. Hence $\kappa_{\beta}$ is null homotopic, implying that 
$\Omega w_{\beta}\circ\kappa_{\beta}$ is null homotopic. 
\smallskip 

%

\noindent 
\textit{Conclusion}. 
The map $\widetilde{f}$ is a product of the maps $\Omega w_{\beta}\circ\kappa_{\beta}$. 
The three cases combine to show that the only factors that are nontrivial 
occur when $2n+m=d(\beta)$, and in these cases the factors are of the form 
\(\llnameddright{S^{d(\beta)}}{E}{\Omega S^{d(\beta)+1}}{\Omega [\nu_{i},\nu_{j}]}{\Omega\overline{\zk}}\) 
where $S^{n_{i}}\times S^{n_{j}}$ appears as a factor in the connected sum $M$ 
satisfying $H^{\ast}(\zk)\cong H^{\ast}(M)$. Thus, after taking adjoints, 
$f$ is homotopic to the sum of the Whitehead products 
\(\llnamedright{S^{d(\beta)+1}}{[\nu_{i},\nu_{j}]}{\overline{\zk}}\) 
that appear in the attaching map for $M$. In other words, $f$ is homotopic to the 
attaching map for $M$. Hence $\zk\simeq M$. 
\end{proof}

\section{Extending to generalized moment-angle complexes} 
\label{sec:gmam} 

This section proves Theorem~\ref{maingmam} by discussing how the proof of Theorem~\ref{main} 
extends to generalized moment-angle complexes $(D^{k},S^{k-1})^{K}$ for $k\geq 2$. This 
begins with the preliminary material in Section~\ref{sec:neighbourly}. 

If $K$ is a triangulation of $S^{2n+1}$ on the vertex set $[m]$ then by~\cite[4.2.11]{BP} 
$(D^{k},S^{k-1})^{K}$ is a manifold of dimension $2n+m(k-1)+2$. Let $\overline{(D^{k},S^{k-1})^{K}}$ 
be the $(2n+m(k-1)+1)$-skeleton of $(D^{k},S^{k-1})^{K}$, so there is a homotopy cofibration 
\[\nameddright{S^{2n+m(k-1)+1}}{g}{\overline{(D^{k},S^{k-1})^{K}}}{}{(D^{k},S^{k-1})^{K}}\] 
where $g$ attaches the $(2n+m(k-1)+2)$-cell.  

There is an analogue of Theorem~\ref{BBCGsplit} for all polyhedral products in~\cite{BBCG1}, 
not just $\Sigma\zk$. In the case of generalized moment-angle complexes it is a natural 
homotopy equivalence that takes the form 
\[\Sigma(D^{k},S^{k-1})^{K}\simeq\bigvee_{I\notin K}\Sigma^{2+(k-1)\vert I\vert}\vert K_{I}\vert.\] 
In the $\zk$ case each wedge summand $\Sigma^{2+\vert I\vert}\vert K_{I}\vert$ 
is more precisely written as $\Sigma^{2} (S^{1})^{\wedge\vert I\vert}\wedge\vert K_{I}\vert$, where 
$(S^{1})^{\wedge\vert I\vert}$ means take the smash product of $\vert I\vert$ copies of $S^{1}$. 
In the case of $(D^{k},S^{k-1})^{K}$, each~$S^{1}$ is replaced by $S^{k-1}$. 

The arguments proving Theorem~\ref{ZK/i} in~\cite{ST} and Proposition~\ref{partialdesuspend}  
both go through without change when $\zk$ is replaced by $(D^{k},S^{k-1})^{K}$, Consequently, 
if $K$ is a neighbourly 
triangulation of $S^{2n+1}$ on the vertex set $[m]$ then the BBCG decomposition for 
$\Sigma (D^{k},S^{k-1})^{K\backslash i}$ desuspends for all $i\in [m]$ and there is a homotopy 
equivalence 
\[\overline{(D^{k},S^{k-1})^{K}}\simeq\bigvee_{I\notin K,I\neq [m]}\Sigma^{1+(k-1)\vert I\vert}\vert K_{I}\vert\] 
where each $\vert K_{I}\vert$ is either contractible or homotopy equivalent to 
$\bigvee_{i=1}^{t_{I}} S^{n}$ for some finite positive integer $t_{I}$. 
The alternative homotopy equivalence for $\overline{\zk}$ in Proposition~\ref{ZKequivs} now follows 
for the $(D^{k},S^{k-1})^{K}$ case by the same argument, noting that the retraction property for 
moment-angle complexes with respect to full subcomplexes in~\cite{DS} holds, in fact, for all 
polyhedral products. 

Theorem~\ref{neighbourlyinW} is a property of $K$ rather than $\zk$. Thus the analogue 
of Corollary~\ref{neighbourlysupport} states that there is a homotopy equivalence 
\[\overline{(D^{k},S^{k-1})^{K}}\simeq\bigvee_{k=1}^{\ell} S^{n_{k}}\] 
where each $S^{n_{k}}$ retracts off a $\mathcal{Z}_{K_{I_{k}}}$ for some $I_{k}$ with 
$(k-1)\vert I_{k}\vert=n_{k}-n-1$. 

The ring isomorphism between the cohomology of $\zk$ and that of a connected sum of products 
of two spheres in Proposition~\ref{neighbourlycohlgy} can be appropriately modified. The 
identification of the cup product in terms of the full subcomplexes $K_{I}$ that appear 
in the decomposition of $\Sigma\zk$ holds in an analogous way for $(D^{k},S^{k-1})^{K}$ 
by replacing the reference to~\cite[Theorem 4.5.8]{BP} with one to~\cite[Theorem 1.9]{BBCG2}.  
The remainder of the argument stays the same, implying that there is a ring isomorphism 
between $H^{\ast}((D^{k},S^{k-1})^{K})$ and the cohomology of a connected sum of products 
of two spheres. Example~\ref{Stasheffexample} then modifies in the same way. 

\begin{proof}[Proof of Theorem~\ref{maingmam}] 
The argument in Section~\ref{sec:proof} for $\zk$ is entirely homotopy theoretic and applies 
equally to $(D^{k},S^{k-1})^{K}$ as it does to $\zk$, once the connectivity and dimension conditions 
used in the proof of Theorem~\ref{main} are modified appropriately. 

The dimensions of the spheres in the decomposition for $\overline{(D^{k},S^{k-1})^{K}}$ 
are different but will still be written as $\bigvee_{k=1}^{\ell} S^{n_{i}}$ for some $n_{i}$. 
The comparison between $2n+m$ and $d(\beta)$ in the cases for the proof of 
Theorem~\ref{main} now become a comparison between $2n+m(k-1)$ and $d(\beta)$, 
where the degree of $d(\beta)$ has changed corresponding to the degree change for 
the $n_{i}$'s. The arguments for Cases 1, 2 and 3(i) are now identical.  
For Case 3(ii), equation~(\ref{wbetadegree}) is replaced by $(\sum_{i=1}^{r} n_{i})-r<2n+m(k-1)$ 
and~(\ref{wbetavertices}) is replaced by $\frac{1}{k-1}((\sum_{i=1}^{r} n_{i})-2(n+1))$. Their 
comparison in~(\ref{wbetafullsc}) becomes 
$\frac{1}{k-1}((\sum_{i=1}^{r} n_{i})-2(n+1))<m(k-1)$. 
The remainder of the argument for Case 3(ii) is now the same. 
\end{proof}

\bibliographystyle{amsalpha}

\end{document}